\documentclass{amsart}
\usepackage{url}
\usepackage[english]{babel}
\usepackage{amsfonts,amssymb,amsmath}

\usepackage{mathrsfs} 
\usepackage{color}

\newcommand{\erre} {{\mathbb {R}}}
\newcommand {\elle} {{\mathscr {L}}}
\newcommand {\elleo} {{\mathscr{L}_0}}

\def\erren{{\erre^{ {N} }}}
\def\erreu{{\erre^{ {N+1} }}}
\def\av{``}
\def\cv{''}
\def\inn{\mbox{ in }}

\def\andd{ \quad\mbox{ and } \quad }

\newcommand{\tende}{\rightarrow}
\newcommand{\ttende}{\longrightarrow}
\newcommand{\enne} {\mathbb{N}}

\newcommand{\meno} {\setminus}

\def\R{\mathbb R}  
\def\hh{{\vskip 1mm \noindent }}
\def\vv{{\vskip 1mm}}
\def\P{{\mathbb P}}
\def\E{{\mathbb E}}

\newtheorem{theorem}{Theorem}[section]

\newtheorem{lemma}[theorem]{Lemma}

\theoremstyle{remark}
\newtheorem{remark}[theorem]{Remark}
\theoremstyle{definition}

\theoremstyle{remark}

\numberwithin{equation}{section}

\title[Harnack inequality and Liouville-type theorems ]{ Harnack inequality and Liouville-type theorems \\for\\ Ornstein--Uhlenbeck and  Kolmogorov operators}

\author{Alessia E. Kogoj}
\address{Dipartimento di Scienze Pure e Applicate (DiSPeA)\\ 
				 Universit\`{a} degli Studi di Urbino Carlo Bo\\
				 Piazza della Repubblica 13, 61029 Urbino (PU), Italy.}
\email{alessia.kogoj@uniurb.it}

\author{Ermanno Lanconelli}
\address{Dipartimento di Matematica\\ 
				 Alma Mater Studiorum Universit\`{a} di Bologna\\
				 Piazza di Porta San Donato 5, 40126 Bologna, Italy.}
\email{ermanno.lanconelli@uniurb.it}

\author{Enrico Priola}
\address{Dipartimento di Matematica\\ 
				 Universit\`{a} degli Studi di Pavia\\
				 Via Adolfo Ferrata 5, 27100 Pavia, Italy.}
\email{enrico.priola@unipv.it}

\subjclass[2010]{35H10; 35B53; 35R03; 35J15; 35K10}
\keywords{Liouville theorems, Harnack inequalities, Ornstein--Uhlenbeck operators, Kolmogorov operators}

\begin{document} 
 
 \dedicatory{\it Dedicato a Sandro Salsa con stima ed amicizia }

\begin{abstract}
We prove, with a purely analytic technique,  a one-side Liouville theorem for a  class
of Ornstein--Uhlenbeck operators ${\mathcal L_0}$ in $\mathbb{R}^N$, as a consequence of a
 Liouville theorem at ``$t=- \infty$"  for the corresponding  Kolmogorov operators
${\mathcal L_0} - \partial_t$ in $\mathbb{R}^{N+1}$. In turn, this last result is proved as a corollary of  
a global Harnack inequality for non-negative solutions to $({\mathcal L_0} - \partial_t) u = 0$ which seems to have
an independent interest in its own right.  
 We stress that our Liouville theorem for ${\mathcal L_0}$ cannot be obtained by a probabilistic approach based on recurrence if $N>2$.
 We provide a self-contained proof of  a Liouville theorem involving recurrent Ornstein--Uhlenbeck stochastic processes  in the Appendix.
\end{abstract}
  
\maketitle

\section{Introduction and main results}

The main \av motivation\cv of this paper is to provide a purely analytical proof of a {\it one-side} Liouville theorem for the following Ornstein--Uhlenbeck operator  in $\erre^N$: 
 
 \begin{equation}\label{OU}  \mathscr{L}_0 :  = \varDelta + \langle Bx,  \nabla \rangle, \end{equation}
 where  $\varDelta$ is the Laplace operator, while $\langle\ , \ \rangle$  and $\nabla$ denote, respectively, the inner product 
 and the gradient in $\erre^N$. Moreover $B$ is a $N\times N$ real matrix which we suppose to satisfy the following conditions: letting 
 
 \begin{eqnarray}\label{1.2} E(t):=\exp(-tB),   \end{eqnarray}   then, 
 
  \begin{equation}\tag{H}\label{H} b:=\sup_{t\in\erre} \| E(t)\| < \infty.  \end{equation}  
  
  It is not difficult to show that condition \eqref{H} is equivalent to the following one: \begin{eqnarray*} &\mbox{\em B is diagonalizable over the complex field}& \\ &\mbox{ \em with all the eigenvalues on the imaginary axis.} & 
 \end{eqnarray*}  
  
  This condition is satisfied in particular if $B=-B^T$ and if $B^2=-\mathbb{I}_N$, where $\mathbb{I}_N$ is the $N\times N$ identity matrix. 
  
   Our positive  (one-side) Liouville Theorem for \eqref{pa} is the following one. 
 \begin{theorem}\label{pa} Let $v$ be a smooth\footnote{$\elleo$ is hypoelliptic, so that every distributional solution to $\elleo u =0$ actually is of class $C^\infty$.} solution to $$\elleo v=0\inn\erren.$$
 If $\inf_\erren v >-\infty$,  then  $v$ is constant. 
 \end{theorem}
 
If we assume the solution $u$ to be bounded both from below and from above then the conclusion of Theorem \ref{pa} immediately follows from a theorem due to Priola and Zabczyk \cite[Theorem 3.1]{priola_zabczyk}, which, for the operator $\elleo$ in    \eqref{OU}, takes this form:\\
{\em Consider the Ornstein--Uhlenbeck operator 
$$\elleo=  \varDelta + \langle Bx,  \nabla \rangle,$$ where $B$ is any $N\times N$ constant matrix. Then the following statements are equivalent:
\begin{itemize}
\item[$(i)$] $\elleo$ has the simple Liouville property, i.e., 
$$\elleo v=0\inn\erren,\ \sup_\erren |v| < \infty \implies v\equiv \mathrm{constant};$$ 
\item[$(ii)$] the real part of every eigenvalue of the matrix B is non-positive.
\end{itemize}   
} 
If the matrix $B$ satisfies \eqref{H}, its eigenvalues have real part equal to zero. Then,
 the aforementioned Priola and Zabczyk theorem implies that the {\em bounded} solutions to $\elleo v=0$ in $\erren$ are constant.

Theorem \ref{pa} is a Corollary of the following Liouville theorem \av at $t=-\infty$\cv\  for the {\em evolution counterpart} of $\elleo$, i.e., for the Kolmogorov operator in $\erreu=\erre_x^N\times \erre_t$
 \begin{equation}\label{K}  \mathscr{L}:  = \varDelta + \langle Bx, \nabla \rangle-\partial_t. \end{equation}

\begin{theorem}\label{liouvilleainfinito} Let $u$ be a smooth solution to $$\elle u=0\inn\erreu.$$
 If $\inf_\erren u >-\infty$,  then  $$\lim_{t\tende -\infty} u(x,t)=\inf_\erreu u \quad \mbox{for every\ } x\in \erren.$$  \end{theorem}

It easy to show that this theorem implies Theorem \ref{pa}. Indeed, let $v:\erren\ttende\erre$ be a smooth and bounded below solution to $\elleo v=0$ in $\erren$. Then, letting 
$$u(x,t)=v(x),\quad x\in \erren,\quad t\in \erre,$$ we have
$$\elle u=0\inn\erreu\andd \inf_\erreu u=\inf_\erren v>-\infty.$$ 
Then, by Theorem \ref{liouvilleainfinito}, 
$$\inf_\erren v = \inf_\erreu u = \lim_{t\tende -\infty} u(x,t) = v(x) \quad \mbox{for every\ } x\in \erren.$$
Hence, $v$ is constant.

From Theorem \ref{liouvilleainfinito} it also follows a Liouville theorem for {\em bounded} solutions to \mbox{$\elle u=0$} (for a related result see Theorem 3.6 in \cite{priola_wang_2006}).  

\begin{theorem}\label{liouvillesemplice} Let $u$ be a bounded smooth solution to $$\elle u=0\inn\erreu.$$
Then, $u$ is constant.  
\end{theorem}
\begin{proof} Let $$m=\inf_\erreu u  \andd  M=\sup_\erreu u.$$ Applying Theorem \ref{liouvilleainfinito} to  $M-u$ and $u-m$, we obtain for every $x\in\erren$ that:

$$0=\inf_\erreu(M-u)=\lim_{t\to -\infty} (M-u(x,t))$$
and 

$$0=\inf_\erreu(u-m)=\lim_{t\to -\infty} (u(x,t)-m).$$
Hence, $M=m$ and $u$ is constant. 
 \end{proof} 
Theorem \ref{liouvilleainfinito} is, in turn, a consequence of a \av global\cv\ Harnack inequality for non-negative solutions to $\elle u=0$ in $\erreu.$ To state this inequality we need to recall that $\elle$ is left translation invariant on the Lie group $\mathbb{K}=(\erreu,\circ)$ with composition law 
$$(x,t)\circ (y,\tau)=(y+E(\tau)x, t+\tau),$$  
see \cite{lanconelli_polidoro_1994}. For every $z_0$ in $\erreu$ we define the \av paraboloid\cv $$P(z_0)=z_0\circ P,$$ 
where
$$P=\left\{ (x,t)\in \erreu \ : \ t< - \frac{|x|^2}{4}\right\}.$$

Then, inspired by an idea used in  \cite{garofalo_lanconelli_1989}  for classical parabolic operators, and exploiting
  Mean Value formulas for solutions to $\elle u=0$,  we establish the following Harnack inequality.

\begin{theorem}\label{harnack}Let $z_0\in \erreu$ and let $u$ be a non-negative smooth solution to $$\elle u=0 \inn \erreu.$$
Then, there exists a positive constant $C$, independent of $u$ and $z_0$, such that 
$$u(z)\le C u(z_0),$$ 
for every $z\in P(z_0).$
\end{theorem} 
We will prove this theorem in Section \ref{proof_harnack}. Here we show how it implies Theorem \ref{liouvilleainfinito} by using the following lemma
(for the  reader's  convenience we postpone its proof to  Section \ref{proof_lemma}).
\begin{lemma}\label{claim} For every $x\in\erren$ and for every $z_0\in\erreu$ there exists a real number $T=T(x,z_0)$ such that 
$$(x,t)\in P(z_0)\qquad \forall t <T.$$ 
\end{lemma} 
\begin{proof}[Proof of Theorem \ref{liouvilleainfinito}] Let $u$ be a smooth bounded below solution to $\elle u=0$ in $\erreu.$
Define $$m=\inf_\erreu u.$$ Then, for every $\varepsilon>0$, there exists $z_\varepsilon \in \erreu$ such that 
$$ u(z_\varepsilon) - m < \varepsilon.$$
Theorem \ref{harnack} applies to  function $u - m$, so that

\begin{equation}\label{pallino} u(z) - m < C ( u(z_\varepsilon) - m) < C\varepsilon,\end{equation} 
for every $z\in P(z_\varepsilon),$ where $C>0$ does not depend on $z$ and on $\varepsilon$. Let us now fix $x\in\erren.$ By Lemma \ref{claim}, there exists $T=T(z_\varepsilon,x)\in\erre$ such that $(x,t)\in P(z_\varepsilon)$ for every $t<T.$ Then, from \eqref{pallino}, we get 
$$ 0\le u(x,t) - m \le C\varepsilon\qquad \forall t<T.$$
This means 
$$\lim_{t\to -\infty} u(x,t) =m.$$

\end{proof} 

We conclude the  introduction with the following remark.
\begin{remark}  One-side Liouville theorems for a class of  Ornstein--Uhlenbeck 
operators can be proved by a probabilistic approach based on recurrence
of the corresponding Ornstein--Uhlenbeck  process. We present this approach in  Appendix,
showing how it leads to one-side Liouville theorems 
also for  degenerate Ornstein--Uhlenbeck operators. 
However,  the  results obtained with this  probabilistic approach
contain Theorem  \ref{OU} only in the case $N = 2.$ We mention that, in this last case, Theorem  \ref{OU} is contained
in \cite{cranston_1983}, where a full description of the 
 Martin boundary for a non-degenerate two-dimensional Ornstein--Uhlenbeck operator  is given.

We also mention that under particular assumptions on the matrix $B$ that make the operator $\elle$ homogenous with respect to a group  of dilations,  asymptotic Liouville theorems at $t=-\infty$ for the solutions  to $\elle u=0$ in $\erreu$ are known (see \cite{kogoj_lanconelli_2007} and the references therein); as a consequence, in such cases, one-side Liouville theorems  for the solutions  to $\elle_0 v=0$ hold. 

  \end{remark} 

\section{Some preliminaries} 
\subsection{} The matrix 

$$E(\tau)=\exp(-\tau B),\quad \tau\in\erre,$$
introduced in \eqref{1.2}, plays a crucial r\^ole for the operator $\elle.$ First of all, as already recalled in the Introduction, defining the composition law $\circ$ in $\erreu$ as follows:

\begin{equation}\label{cl}  (x,t)\circ (y,\tau)=(y+E(\tau)x, t+\tau), \end{equation} 
we obtain a Lie group 

$$\mathbb{K}=(\erreu, \circ),$$
on which $\elle$ is left translation invariant (see \cite{lanconelli_polidoro_1994}; see also \cite{BLU}, Section 4.1.4).

As already observed, assumption \eqref{H} implies 
$$\sigma(B):=\ \{ \mbox{eigenvalues of $B$} \} \subseteq i \erre.$$

Then, since $B$ has real entries, $-\lambda\in\sigma(B)$ if $\lambda \in \sigma(B).$ As a consequence, 
$$\mathrm{trace\, }(B)=0. $$
\subsection{}  A fundamental solution for $\elle$ is given by 

\begin{equation}\label{2} \Gamma(z,\zeta) = \gamma (\zeta^{-1} \circ z),
\end{equation}
where, 
\begin{equation*} \gamma(z)=\gamma (x,t)= \begin{cases} 0 \mbox{\ if \ } t\le 0, \\  \\  \dfrac{(4\pi)^{-\frac{N}{2}} }{\sqrt{\det C(t)}Ê} \exp\left(-\dfrac{1}{4} \langle C^{-1}(t) x,x \rangle \right)   \mbox{\ if \ } t> 0,
 \end{cases} 
\end{equation*} and

\begin{equation*} C(t) = \int_0^t E(s)E(s)^T\ ds,\end{equation*} 

(see \cite[(1.7)]{lanconelli_polidoro_1994}, and keep in mind that $\mathrm{trace\,}Ê (B)=0$ since $B$ satisfies \eqref{H}).

It is noteworthy to stress that 
$$C(t) \mbox{ is symmetric and }ÊC(t)>0$$
for every $t>0.$

\subsection{} The solutions to $\elle u=0$ in $\erreu$ satisfy the following Mean Value formula: for every $z_0\in\erreu,\ r>0$ and $p\in\enne,$ 

\begin{equation}\label{mv}  u(z_0)= \frac{1}{r} \int_{{\Omega}_r^{(p)}(z_0)} u(z) W_r^{(p)} (z_0^{-1} \circ z)\ dz,
\end{equation} 
where

$$\Omega_r^{(p)} (z_0)= \left\{ z \ : \ \phi_p(z_0,z)> \frac{1}{r} \right\},$$ 
with 

$$\phi_p(z_0,z): = \frac{\Gamma(z_0,z)}{(4\pi(t_0-t))^{\frac{p}{2}}},$$
if $z=(x,t)$ and $z_0=(x_0,t_0).$ 

\begin{remark} If $z\in \Omega_r^{(p)}(z_0)$, then $\Gamma(z_0,z) >0$, hence $t_0-t>0.$  \end{remark} 

Moreover, 
\begin{equation}\label{2.4} W_r^{(p)} (z)= \omega_p R_r^p (0,z) \left\{ W(z) + \frac{p}{4(p+2)} \left( \frac{R_r(0,z)} {t}\right)^2\right\}, 
\end{equation} 
where $\omega_p$ denotes the Lebesgue measure of the unit ball of $\erre^p,$
\begin{equation}\label{2.5} W(z)= W(x,t )= \frac{1}{4} \left| C^{-1}(t)x\right|^2\!\!\!\!,
\end{equation} 
 and  
\begin{equation}\label{2.6} R_r(0,z)=\sqrt{ 4(-t)\log(r\phi_p(0,z))}.
\end{equation} 
A complete  proof of the Mean Value formula \eqref{mv} can be found in Section 5 of  \cite{cupini_lanconelli_2020}.

\section{Proof of Lemma \ref{claim}}\label{proof_lemma} 
Let $z_0=(x_0,t_0)$ and $z=(x,t).$ Then, 
$$ z\in P(z_0)= z_0\circ P \iff z_0^{-1} \circ z \in P \iff (x- E(t-t_0) x_0, t- t_0)\in P.$$
Hence, keeping in mind the definition of $P$, 
\begin{equation} \label{4} z=(x,t)\in P(z_0) \iff \frac{{|x-E(t-t_0) x_0|}^2}{4(t_0- t)} <1.
\end{equation}
On the other hand, from \eqref{H}, we have 

\begin{equation*}  \frac{{|x-E(t-t_0) x_0|}^2}{4(t_0- t)} \le \frac{(|x|+ b |x_0|)^2}{4(t_0-t)}\ttende 0,\quad \mbox{as } t  \ttende -\infty.
\end{equation*}
Therefore: for every fixed $z_0\in\erreu$ and $x\in\erre$, there exists $T=T(z_0, x)$ s.t. 
$$z=(x,t)\in P(z_0)\qquad \forall\ t<T.$$ 

\section{ A two \av onions\cv \  lemma}

The aim of this section is to prove a geometrical lemma on the level sets $\Omega_r^{(p)}$ (which we call {\em $\elle$-\av onions\cv}), that will play a crucial  r\^ole in the proof of the Harnack inequality in Theorem \ref{harnack}. 

First of all we resume that hypothesis \eqref{H} implies:
\begin{equation}Ê\label{1.4} \frac{1}{b^2} |x|^2 \le t \langle C^{-1} (t) x, x \rangle \le b^2 |x|^2, 
\end{equation}Ê
for every $t\in\erre$ and for every $x\in\erre^N.$

Indeed, from \eqref{H},  we obtain 
$$
b:=\sup_{t\in\erre} \| E(t)^T\| < \infty.  
$$
Since we are considering the operator norm, we have
$$
|E(s)^T y| \le b |y| = b|E(-s)^T E(s)^T y|\le b^2 |E(s)^T  y|,$$
so that 
$$\frac{1}{b} |y| \le |E(s)^T  y| \le b|y|$$
for every $t\in\erre$ and every $y\in\erren$.

Then, since 
\begin{equation*} \langle C(t) y, y \rangle  = \int_0^t |E(s)^T y|^2\ ds,\end{equation*} 
we get 
$$\frac{1}{b^2} |y|^2 \le \frac{1}{t} \langle C(t) y, y \rangle \le b^2|y|^2$$
for every $y\in\erren$ and $t\in \erre\meno \{  0\}.$ If in these inequalities we choose 

$$y=(C(t))^{-\frac{1}{2}}x \mbox{\qquad   if \ } t>0$$
and 
$$y=(-C(t))^{-\frac{1}{2}}x \mbox{\qquad   if \ } t<0,$$
we immediately obtain \eqref{1.4}.

Now, for every $r>0$, define 

$$\Sigma_r = \left \{ z=(x,t) \ : \ t=- r^{\frac{2}{N+p}},\ |x|^2 < -4 t \right \}.$$ 

Then, the following lemma holds 

\begin{lemma}\label{4uno}   For every $p\in\enne$, there exists a constant $\theta=\theta(p) >1$ such that, 
$$\Omega_{\theta r}^{(p)} (0) \supseteq \Omega_r^{(p)}(z) \quad \forall\,  z\in \Sigma_r, \quad \forall\, r>0.
$$ 
\end{lemma} 
\begin{proof}
Let $r>0$ and $z\in \Sigma_r.$ Then $z=(x,t),$ with 

$$t=-r^{\frac{2}{N+p}}\andd  |x|^2 < 4 r^{\frac{2}{N+p}}.$$

Let us now take $\zeta=(\xi,\tau) \in  \Omega_r^{(p)}(z).$ This means

\begin{eqnarray}\label{4.2}\nonumber &\phi_p(z,\zeta) > \dfrac{1}{r} &
\\ & \iff& \\\nonumber & \langle C^{-1}(t-\tau) ( x - E(t-\tau)\xi), x - E(t-\tau)\xi\rangle   <  \log \dfrac{r} {(4\pi (t-\tau))^{\frac{N+p}{2}} }. &
\end{eqnarray} 

Analogously, 

\begin{eqnarray*}\nonumber & \zeta \in  \Omega_{\theta r}^{(p)}(0) &
\\ & \iff& \\\nonumber & \langle C^{-1}(-\tau) E(-\tau)\xi, E(-\tau)\xi\rangle   <  \log \dfrac{\theta r} {(4\pi (-\tau))^{\frac{N+p}{2}} }. &
\end{eqnarray*} 
On the other hand, by \eqref{1.4} and \eqref{H}, 

\begin{eqnarray*} \langle C^{-1}(-\tau) E(-\tau)\xi, E(-\tau)\xi\rangle  \le  b^4 \dfrac{|\xi|^2}{|\tau|}, 
\end{eqnarray*} 
so that, $\zeta=(\xi,\tau)  \in  \Omega_{\theta r}^{(p)}(0) $ if $\tau<0$ and 

\begin{eqnarray}\label{4.3} |\xi|^2    <  \dfrac{1}{ b^4} |\tau| \log \dfrac{\theta r} {(4\pi |\tau|)^{\frac{N+p}{2}} }. 
\end{eqnarray} 
Then, to prove our lemma, it is enough to show that inequality \eqref{4.2} implies \eqref{4.3}. Now, from \eqref{4.2},  using 
\eqref{H}, \eqref{1.4}  and the inclusion  $z=(x,t)\in \Sigma_r$, we obtain 
 (we assume $b \ge 1$ so that $b^2 \le b^4$)
\begin{eqnarray*}|\xi|^2&\le& b^2 |E(t-\tau)\xi|^2  \\ &\le& 2 b^2 (|E(t-\tau)\xi - x|^2 +|x|^2) \\ &\le& 2 b^4
\left( (t-\tau) \langle C^{-1}(t-\tau) ( E(t-\tau)\xi-x), E(t-\tau)\xi-x \rangle  + 4|t| \right) \\ &<&  2 b^4 \left( (t-\tau) \log \dfrac{r} {(4\pi (t-\tau))^{\frac{N+p}{2}} }+ 4|t| \right) . 
\end{eqnarray*} 
Therefore, we will obtain  \eqref{4.3}, and hence the lemma, if for a suitable $\theta >1$ independent of $z$ and $\zeta$,  the following inequality holds
\begin{eqnarray}\label{4.4}Ê 2 b^4 \left( (t-\tau) \log \dfrac{r} {(4\pi (t-\tau))^{\frac{N+p}{2}} }+ 4|t| \right) \le 
\frac{1}{b^4}  |\tau| \log \dfrac{\theta r} {(4\pi |\tau|)^{\frac{N+p}{2}} }.
\end{eqnarray}

To simplify  the notation we put 

$$  \frac{ r} {(4\pi )^{\frac{N+p}{2}}} = \rho^{\frac{N+p}{2}} \iff \rho =  \frac{ r^{\frac{2}{N+p}} } {4\pi }.  $$
Hence, since  $z\in \Sigma_r$, 
$$|t|=4\pi\rho,$$
and inequality \eqref{4.4} can be written as follows:

\begin{eqnarray}\label{4.5} A_0  (t-\tau) \log \frac{\rho}{t-\tau}+ A_1\rho \le  A_2  |\tau| \log  \frac{\theta \rho}{|\tau|},\end{eqnarray} 
and the $A_i$'s are strictly positive constants independent of $z$ and $\zeta$. 

Since $\zeta \in \Omega_{r}^{(p)}(z)$, we have  
\begin{eqnarray*} \frac{1}{r} < \phi_\rho(z,\zeta) \le \left(\frac{1}{4 \pi (t-\tau)} \right)^{\frac{N+p}{2}}, 
\end{eqnarray*} 
then, 
$$ 0<t-\tau<\rho.$$ 
As a consequence, since 
$$ 4 \pi\rho=|t|<|\tau|\le |\tau - t| +|t|<\rho +4\pi\rho,$$ we get 

$$ \frac{1}{4\pi+1} \le \frac{\rho}{|\tau|}\le  \frac{1}{4\pi}. $$

Thus, the left hand side of \eqref{4.5} can be estimated from above as follows:

\begin{eqnarray*} A_0 (t-\tau) \log \frac{\rho}{t-\tau}+ A_1 \rho = \rho \left( A_o  \frac{t -\tau }{\rho}\log \frac{\rho}{t-\tau} +A_1\right) \le \rho(A_0S+A_1), 
\end{eqnarray*}

where $$S=\sup \left\{ s \log\frac{1}{s} \ : \ 0<s<1 \right\}.$$
Moreover,  the right hand side of \eqref{4.5} can be estimated from below as follows: 
\begin{eqnarray*} A_2 |\tau |  \log  \frac{\theta \rho}{|\tau|} \geq  \rho 4\pi A_2 \log  \frac{\theta}{4\pi+1}. \end{eqnarray*} 
Therefore, if we choose $\theta>0$ such that 
$$A_0S+A_1\le 4\pi A_2 \log  \frac{\theta}{4\pi+1}$$
inequality  \eqref{4.5}  is satisfied. This completes the proof. 

\end{proof} 

\section{Proof of Theorem \ref{harnack}}\label{proof_harnack} 

Since $\elle$ is left translation invariant on the Lie group $(\mathbb{K}, \circ)$, it is enough to prove Theorem  \ref{harnack} in the case $z_0= 0 \in \R^{N+1}.$ In particular, it is enough to prove the inequality 
 
 \begin{eqnarray}\label{5uno} u(z)\le C u(z_0), \mbox{\qquad with \ } z_0=0,\end{eqnarray} 
 for every non-negative smooth solution $u$ to 
 
 $$\elle u=0\inn \erreu,$$
 and for every  $z= (x,t) \in  P=\{(x,t) \  : \ |x|^2 < -4t \}.$
 
 The constant $C$ in \eqref{5uno} has to be independent of $u$.  To this end, taken a non-negative global solution $u$ to $\elle u=0$, we start with the Mean Value formula for $u$ on the $\elle$-level set  ${\Omega}_{2\theta r}^{(p)}(z_0)$, with $p>4$ and with $\theta$ given by Lemma \ref{4uno}:

\begin{equation}\label{5due}  u(z_0)= \frac{1}{2 \theta r} \int_{{\Omega}_{2\theta r}^{(p)}(z_0)} u(\zeta) W_{2\theta r}^{(p)} (z_0^{-1} \circ \zeta)\ d\zeta.
\end{equation}

Let us arbitrarily fix $z=(x,t) \in P.$ Then $t<0$ and $|x|^2<4|t|.$ In \eqref{5due} we choose $r>0$ such that 
$$t=-r^{\frac{2}{N+p}}.$$ 
By Lemma  \ref{4uno} we have the inclusion $${\Omega}_{2\theta r}^{(p)}(z_0)\supseteq {\Omega}_{ r}^{(p)}(z), $$
so that, since $u\geq 0$, from \eqref{5due} we get 

\begin{equation}\label{5tre}  u(z_0)\geq  \frac{1}{2 \theta r} \int_{{\Omega}_{r}^{(p)}(z)} u(\zeta) W_{2\theta r}^{(p)} (z_0^{-1} \circ \zeta)\ d\zeta.
\end{equation}

Let us now prove that, for a suitable positive constant $C$ independent of $u$ and of $z$, we have ($z_0^{-1}=z_0=0$):

\begin{equation}\label{5quattro}  \frac{ W_{2\theta r}^{(p)} (z_0^{-1} \circ \zeta)}{ W_{ r}^{(p)} (z^{-1} \circ \zeta)}\geq \frac{2\theta}{ C}\qquad\forall \zeta\in {\Omega}_{r}^{(p)}(z).
\end{equation} 
It will follow,  from \eqref{5tre},      
\begin{eqnarray*}  u(z_0)&\geq&  \frac{1}{r C} \int_{{\Omega}_{r}^{(p)}(z)} u(\zeta) W_{r}^{(p)} (z^{-1} \circ \zeta)\ d\zeta\\
&\phantom{=}&\mbox{(again by the Mean Value formula \eqref{mv})}\\
&=&  \frac{1}{C} u(z),
\end{eqnarray*}  
i.e., $u(z)\le C u(z_0),$ which is \eqref{5uno}.

To prove \eqref{5quattro} we first estimate from below $W_{2\theta r}^{(p)} (z_0^{-1} \circ \zeta)$. From the very definition of this kernel, by keeping in mind that $z_0= 0$, and letting $\zeta=(\xi, \tau),$ we obtain:
\begin{eqnarray*} W_{2\theta r}^{(p)} (z_0^{-1} \circ \zeta) &\geq& \frac{p\omega_p}{4(p+2)}  \frac{(R_{2\theta r}(z_0,\zeta))^{p+2}}{|\tau|^2} 
\\  &=& c'_p |\tau|^{\frac{p+2}{2}-2} (\log(2\theta r \phi_p(z_0,\zeta )))^{\frac{p}{2}+1} 
\\ &\phantom{\geq}& \mbox{ ($\phi_p(z_0,\zeta)\geq\frac{1}{\theta r}$ since $\zeta\in {\Omega}_{r}^{(p)}(\zeta)\subseteq {\Omega}_{\theta r}^{(p)}(z_0)$)} 
\\ &\geq& c'_p (\log (2 \theta))^{\frac{p}{2}+1} |\tau|^{\frac{p}{2} -1} \\ &\phantom{\geq}& \mbox{ (if $p>2$)}\\&\geq&  c_p |t|^{\frac{p}{2} -1}
\\  &=& c_p r^{\frac{p-2}{p+N}}.
\end{eqnarray*} 
Here, and in what follows,  $c'_p,c''_p, \ldots, c_p$ denote  strictly positive constants only depending on $p$. So, we have proved the following  inequality 
\begin{eqnarray} \label{5cinque} W_{2\theta r}^{(p)} (z_0^{-1} \circ \zeta) \geq c_p r^{\frac{p-2}{p+N}} \qquad\forall \zeta\in {\Omega}_{r}^{(p)}(z).
\end{eqnarray} 
Now we estimate $W_{ r}^{(p)} (z^{-1} \circ \zeta)$ from above, estimating, separately 
\begin{eqnarray} \label{5sei} K_1(z,\zeta)= R_r^p(0, z^{-1} \circ \zeta) W(z^{-1} \circ \zeta)\end{eqnarray} 
and 
\begin{eqnarray} \label{5sette} K_2(z,\zeta)= \frac{R_r^{p+2} (z_0, z^{-1} \circ \zeta)}{(t-\tau)^2}.\end{eqnarray} 
We have 
\begin{eqnarray}\label{5.8}\nonumber  K_1(z,\zeta)&=& \left( 4(t-\tau) \log \left( r \frac{\Gamma (z,\zeta)}{(4\pi (t-\tau))^{\frac{N+p}{2}}} \right)\right)^{\frac{p}{2}}W(z^{-1} \circ \zeta) )\\ 
&\le  & 2^p  \left( (t-\tau) \log \frac{r}{ (t-\tau)^{\frac{N+p}{2}}} \right)^{\frac{p}{2}} W(z^{-1} \circ \zeta). \end{eqnarray} 
Moreover, from \eqref{2.5} and \eqref{1.4}, we obtain 

\begin{eqnarray}\label{5.9}\nonumber
 W(z^{-1} \circ \zeta)&=& \frac{1}{4} \left| C^{-1} (\tau-t) ( \xi- E(\tau-t) x)\right|^2 \\
 &\le & \frac{b^4}{4} \frac{|\xi- E(\tau-t) x|^2} {(\tau-t)^2}. \end{eqnarray}
To estimate the right hand side of this inequality we use the inclusion  $\zeta\in \Omega^{(p)}_r(z)$ which implies:

\begin{eqnarray*} &\phi_p(z,\zeta) > \dfrac{1}{r} &
\\ & \iff& \\  & \left( \dfrac{1} {(4\pi (t-\tau))}\right)^{\frac{N+p}{2}} \exp\left( -\dfrac{1}{4}  \langle C^{-1}(t-\tau) ( x - E(t-\tau)\xi), x - E(t-\tau)\xi\rangle\right)  > \dfrac{1}{r}  &
\\ & \iff& \\   &\langle C^{-1}(t-\tau) ( x - E(t-\tau)\xi), x - E(t-\tau)\xi\rangle   <  \log \dfrac{r} {(4\pi (t-\tau))^{\frac{N+p}{2}} }. &
\end{eqnarray*} 

This inequality, keeping in mind \eqref{1.4}, implies

\begin{eqnarray*} \left| x - E(t-\tau)\xi\right|^2     \le b^2 (t-\tau)  \log \dfrac{r} {(4\pi (t-\tau))^{\frac{N+p}{2}} }. 
\end{eqnarray*} 

Then
\begin{eqnarray*} \left| \xi - E(\tau-t) x \right|^2     &\le& \left\| E(\tau-t) \right\|^2  \left|  E(t-\tau) \xi -x \right|^2 \\
 &\le&  b^4 (t-\tau)  \log \dfrac{r} {(4\pi (t-\tau))^{\frac{N+p}{2}} }\\
 &\le& c'_p \frac{r^{\frac{2}{N+p}}}{4\pi},
\end{eqnarray*}

where $$c'_p=b^4    \sup \left\{  s \log \frac{1}{s} : 0 <s<1\right\}.$$

Using this estimate in \eqref{5.9} and \eqref{5.8} we obtain:

\begin{eqnarray}\label{5.10} K_1 (z,\zeta) \le  c''_p r^{\frac{2}{N+p}} (t-\tau)^{\frac{p}{2}-1} \left( \log\frac{r}{{(4\pi(t-\tau))}^{\frac{N+p}{2}}}\right) ^{\frac{p}{2}} \le  c_p r^{\frac{p-2}{N+p}},
\end{eqnarray}

where, 
$c_p= c'''_p   S_p$, with 

$$S_p= \sup \left\{ s^{\frac{p}{2} -2} \left( \log \frac{1}{s}\right)^{\frac{p}{2}}   : 0 <s<1\right\}.$$

We stress that $S_p<\infty$ since $p>4$.

The same estimate holds for $K_2$. Indeed:
\begin{eqnarray}\label{5nove} K_2 (z,\zeta) &\le&  c'_p (t-\tau)^{\frac{p}{2}-1} \left( \log\frac{r}{{\left(4\pi (t-\tau)\right)}^{\frac{N+p}{2}}}\right) ^{\frac{p+2}{2}}\\ \nonumber   &\le& c^p r^{\frac{2}{N+p} (\frac{p}{2}-1)} = c_p r ^{\frac{p-2}{N+p}}, \end{eqnarray}

where, 
$$c_p= c'_p  \sup \left\{ s^{\frac{p}{2} -1} \left( \log \frac{1}{s}\right)^{\frac{p+2}{2}}   : 0 <s<1\right\} <\infty.$$

Keeping in mind \eqref{5sei} and  \eqref{5sette}, and the very definition of $W^{(p)}_r(z,\zeta)$, from inequalities \eqref{5.10} and 
 \eqref{5nove} we obtain 
 
\begin{eqnarray} \label{5dieci} W_{r}^{(p)} (z^{-1} \circ \zeta) \le c_p r^{\frac{p-2}{p+N}} \qquad\forall \zeta\in {\Omega}_{r}^{(p)}(z).
\end{eqnarray} 
 This inequality, together with  \eqref{5cinque}, implies  \eqref{5quattro}, and completes the proof of Theorem  \ref{harnack}.

\section{Appendix: A one-side Liouville theorem for  Ornstein--Uhlenbeck operators by recurrence}

Here we show a  one-side Liouville theorem for some Ornstein--Uhlenbeck (OU) operators  based on  recurrence  of the corresponding OU stochastic processes. 

It is a general fact from probabilistic potential theory (see in particular \cite{getoor}) that recurrence of a Markov process is equivalent to the fact that all excessive functions are constants (we also mention that the equivalence between  excessive functions and  super harmonic functions has been established in a general setting; see \cite{dynkin} and the references therein). 
  On the other hand, a  characterization of recurrent   OU processes is known (see \cite{eric} which extends   the seminal paper  \cite{dym}; see also \cite{Z81} for connections between recurrence and stochastic controllability). 

  \vv 
 
 { We present the main steps to prove   a one-side Liouville theorem in a self-contained way.  Comparing with \cite{dym}, \cite{eric} and \cite{getoor}, we simplify some proofs;  see in particular the proof of Theorem \ref{due}
 in which we also use a result in \cite{priola_zabczyk}.  We       
  do not     appeal to  the general theory of  Markov processes but  we use 
   some  basic stochastic calculus. }  
 It seems to be an open problem   
to find a purely analytic approach to proving such result. 
   
\hh 
Let $Q$ be a non-negative symmetric $N \times N$ matrix and let $B$ be a real 
$N \times N$ matrix. The   OU operator we consider is 
\begin{equation}\label{ss}
{\mathcal K}_0     = \frac{1}{2}\text{ tr}  (QD^2 ) + \langle Bx, \nabla \rangle =  \frac{1}{2}\text{ div}  (Q  \nabla ) + \langle Bx, \nabla \rangle.
\end{equation}
We will always assume the well-known Kalman  controllability  condition: 
 \begin{equation} \label{kal}   
{ \text{rank}[{Q}, B {Q}, \ldots, B^{N-1} {Q} ] =N,} 
\end{equation}
see \cite{eric}, \cite{Z81}, \cite{lanconelli_polidoro_1994},  \cite{DZ}, \cite{priola_zabczyk} and the references therein. 
 Under this assumption   ${\mathcal K}_0$
 is hypoelliptic, see \cite{lanconelli_polidoro_1994}. 
  Before stating the  Liouville theorem  we recall  that  a matrix $C$ is stable if all its eigenvalues have negative real part. 
\begin{theorem}\label{ouProb} Assume \eqref{kal}. Let $v: \R^N \to \R$ be a non-negative $C^2$-function such that 
$
{\mathcal K}_0 v \le 0$ on $\R^N$. Then $v$ is constant  if the following condition holds:

\vskip 1mm \noindent {\bf (HR)} The  real Jordan  representation of $B$ is
\begin{equation}\label{ss2}
 \begin{pmatrix}
 B_0 & 0\\
 0 & B_1
\end{pmatrix}
\end{equation}
where $B_0$ is stable and  $B_1$ is at most of dimension 2 and of the form
 $B_1 = [0]$ or 
 $B_1 =  \begin{pmatrix}
 0  & -\alpha \\
 \alpha & 0
\end{pmatrix}
$
 for some $\alpha \in \R$ (in this case we need $N \ge 2$).
   \end{theorem}
  {\sl The proof of Theorem \ref{ouProb} will immediately follow  by Lemma \ref{uno} and Theorem \ref{due} below.}  
   \begin{remark}
Note that when $N=2$ the matrix $B = \begin{pmatrix}
 0  & 1 \\
 0 & 0
\end{pmatrix}$   does not satisfy (HR). On the other hand $B= \begin{pmatrix}
 0  & 0 \\
 0 & 0
\end{pmatrix}$ verifies (HR) with $\alpha =0$. Moreover, 
 an example of possibly degenerate 
  two-dimensional OU operator for which  the one-side Liouville theorem holds is  
\begin{gather*}
{\mathcal K}_0 = \partial_{xx}^2 +  a \partial_{yy}^2 + x \partial_y - y \partial_x,\;\;\;\;\;  a \ge 0.
\end{gather*}
\end{remark}

\begin{remark} It is well-known,  that  condition \eqref{kal} is  equivalent to the fact that 
\begin{equation}\label{qtt2}
Q_t = \int_0^t \exp(sB) \, Q \exp(sB^T) ds \;\;\text{is positive definite for all} \; t>0
\end{equation}
(cf. \cite{eric}, \cite{lanconelli_polidoro_1994} and \cite{DZ}). Note that $C(t) =  \exp(-tB) Q_t  \exp(-tB^T)$ is used in \cite{lanconelli_polidoro_1994}   and in Section 5 of  \cite{cupini_lanconelli_2020} with $Q$ replaced by $A$. 
  \end{remark}  
  Let us introduce the OU stochastic process starting at $x \in \R^N$. It is the solution to  the following  linear SDE  
 \begin{equation}\label{ou34}
{ X_t^x(\omega) =  x + \int_0^t B X_s^x(\omega) ds
   \, + \,
  \sqrt{Q}  \,   W_t(\omega),\, \;\; t \ge 0,\;\; x \in \R^N,\; \omega \in \Omega,} 
 \end{equation} 
see, for instance, \cite{eric} and  \cite{priola_zabczyk}. Here $W = (W_t)$ is a standard $N$-dimensional Wiener process defined  a stochastic basis $( \Omega,  {\mathcal  F}, ({\mathcal  F}_t),  \P)$   (the expectation with respect to $\P$ is denoted by $\E$; as usual in the sequel we often do not indicate  the dependence on $\omega \in \Omega$).     
 
\vv
For any non-empty open set $O \subset \R^N$,  we consider  the {\sl hitting time} $\tau^x_O = \inf \{ t \ge 0 \, :\, X_t^x \in O \}$ (if $\{ \cdot \}$ is empty we write $\tau^x_O  = \infty$). 

\hh 
Now we recall  the notion of recurrence. The OU process  $(X^x_t)_{t \ge 0} = X^x$ is {\sl   recurrent} if for any $x \in \R^N$, for any non-empty open set $O \subset \R^N$,  one has
\begin{equation}\label{sst}
\phi_O (x) = \P (\tau^x_O < \infty ) =1.  
\end{equation}
 Thus recurrence means that with probability one,  the OU process reaches in finite time any open set starting from any initial position $x$.

\begin{lemma} \label{uno} Suppose that the OU process is recurrent.
 Let $v\in C^2(\R^N)$ be a non-negative function such that 
$ {\mathcal K}_0 v \le 0$ on $\R^N$. Then $v$ is constant.
\end{lemma}
\begin{proof} We will adapt an  argument used in the proof of Lemma 3.2 of \cite{getoor}  to show that excessive functions are constant for      recurrent Markov processes. 

Let us fix $x \in \R^N$. Applying the It\^o formula and using the fact that ${\mathcal K}_0 v \le 0$ we get, $\P$-a.s.,   
\begin{gather*}
v(X^x_t) = v(x) + \int_0^t {\mathcal K}_0 v (X_s^x) ds + M_t \le v(x) + M_t,  \;\; t \ge 0,
\end{gather*}
 where we are considering  the martigale $M= (M_t)$,  $M_t = \int_0^t  \nabla v(X_s^x) \cdot \sqrt{Q} dW_s$.
 
 Let $O \subset \R^N$ be a non-empty open set and consider 
the  hitting time $\tau^x_O$.   
 We have  $0 \le  v(X^x_{t \wedge \tau_O^x}) \le  v(x) + M_{t \wedge \tau_O^x}$, $t \ge 0$.  
  By the Doob optional stopping theorem    we obtain  
\begin{gather*}
 \E[ v(X^x_{t \wedge \tau_O^x})] \le v(x),\;\; t \ge 0.
\end{gather*}
Hence 
\begin{gather} \label{d22}
v(x) \ge  \E[ v(X^x_{n \wedge \tau_O^x})] \ge  \E[ v(X^x_{n \wedge \tau_O^x})\, 1_{\{ \tau_O^x < \infty \}}],\;\; x \in \R^N,\;\; n \ge 1.
\end{gather}
Recall that 
$\P (\tau_O^x < \infty ) =1$, for any $x \in \R^N$. 
 By the Fatou lemma (using also the continuity of the paths of the OU process) we infer
\begin{equation}\label{ss1}
  \E[ v(X^x_{\tau_O^x}) ] = 
  \E[\liminf_{n \to \infty}  v(X^x_{n \wedge \tau_O^x}\,)] 
 \le v(x). 
  \end{equation}  
Now we argue by  contradiction.  Suppose that $v$ is not constant. Then there exists $0 < a < b$,  $z \in \R^N$ such that $v(z) <a$ and $U = \{ v >b\}$ $=\{x \in \R^N \, :\,  v(x)  >b\}$ which is a non-empty open set. By \eqref{ss1} with $x = z$ we obtain
\begin{gather*}
a > v(z) \ge    \E \big [ v(X^z_{\tau_{U}^z})
\big ] \ge   b 
\end{gather*}
because on the event $\{ \tau_{U}^z < \infty \}$ we know that $X^z_{\tau_{U}^z} \in \{ v \ge b \}$. 
 We have found the contradiction $a >b$. Thus $v $ is constant. 
  \end{proof}
  Recall the {\sl OU Markov semigroup}  $(P_t)= (P_t)_{t \ge 0}$,
  \begin{gather} \label{d33}
 P_t f(x) = (P_tf)(x) = \E [f(X^x_t)] =   \int_{\R^N} f(y) \, p_t(x,y) dy,\;\; t >0,
 \end{gather}
where $x \in \R^N$, $f: \R^N \to \R$    Borel and bounded and  
  $p_t(x,y) = 
   \frac{ e^{- \frac{ |Q_t^{-1/2} (e^{tB} x - y)|^2}{2}} } { \sqrt{(2 \pi)^N \det(Q_t)} }   \, $. We set $P_0 f= f$. 
  The associated {\sl potential} of a non-negative Borel  function $g: \R^N \to \R$ is 
  \begin{equation}\label{po}
Ug(x) =\int_{0}^{\infty}P_t g(x) dt,\;\; x \in \R^N.
\end{equation} 
  Clearly, in general it can also assume the value $\infty$ (cf. \cite{getoor}). 
  
\begin{remark}  Let $A$ be an 
  empty open set and let  
   $1_A$ be  the indicator function of $A$.   
  The probabilistic interpretation of $U 1_A$ is as follows. First one defines the sojourn time or occupation time of $A$ (by the OU process starting at $x$)  as
 \begin{equation*} 
J_A^x(\omega)= \int_{0}^{\infty} 1_{A}(X_t^x (\omega))dt,\;\; \omega \in \Omega;
 \end{equation*} 
it is the total amount of
time that the sample path  $t \mapsto X^x_{t}(\omega)$  
 spends in $A$. Then 
 $
 \E [J_A^x ] =   
 \int_{0}^{\infty} \E [1_{A}(X_t^x)] dt$ $  = U 1_A  (x)
 $
 is 
 the average sojourn time  or the expected occupation time of $A$. 
 \end{remark} 
 {  The next result  is a reformulation of  a theorem in \cite{eric} at page 822 (see also the comments before such theorem and \cite{dym}). Erickson proves   
  some parts of the theorem and refers to \cite{dym} for the proof of 
  the remaining parts.
 \begin{theorem} \label{due}
Assume \eqref{kal}. The next  conditions for the OU process are equivalent. 

\hh (i) Condition (HR) holds. 

\hh (ii) $\int_{1}^{\infty}  \frac{1}{ \sqrt{ \det(Q_t)} } dt = \infty$. 

\hh (iii) For any $x, y \in \R^N$, 
\begin{equation}\label{dd}
 \int_{1}^{\infty}  p_t(x,y) dt = \infty.
\end{equation}
(iv) The OU process $(X_t^x)$ is recurrent.
\end{theorem}
   We will only deal with    the  proofs of  $(i)  \Rightarrow (ii) \Rightarrow (iii)$ and $(i) \Rightarrow (iv)$; the last implication is   needed to prove the  one-side Liouville theorem in  Lemma \ref{uno}. 
  
  The   proof of the recurrence $(i) \Rightarrow (iv)$
  is different and  simpler than the  proof given in  \cite{dym}  which  also \cite{eric} mentions (see the remark below for more details).
  
\begin{remark}  
  In \cite{dym} it is proved that $(iii)  \Rightarrow (iv) $ by showing first that (iii) implies that,  for any non-empty open set $O$, one has
 $U1_O \equiv  \infty$, and then
    using a quite involved Khasminskii argument 
    (see pages 142-143 in \cite{dym}) which uses  the  strong Markov property,   the   irreducibility and  strong Feller property of the OU process.  Alternatively, the fact that 
 $U1_O \equiv  \infty$, for any non-empty open set $O$, is equivalent to recurrence can be obtained  using a potential theoretical approach involving excessive functions 
 as in \cite{getoor} (see in particular the  proof that (ii) implies (iv) in Proposition 2.4  and Lemma 3.1 in \cite{getoor}).
 \end{remark}
 }
 \begin{proof}
\hh  $ {\bf   (i)  \Rightarrow (ii)}$.    This can be  proved as in the proof of Lemma 6.1 in  \cite{dym} by using the Jordan decomposition of the matrix $B$  (see also the remarks in \cite{eric}).

\hh $ {\bf  (ii)  \Rightarrow (iii)}$   
   Note that $Q_{t} \le Q_{T}$ (in the sense of positive symmetric matrices) if $0< t \le T$. Hence by the Courant-Fischer min-max principle,
we have $\lambda(t) \le \lambda(T)$ (where $\lambda (s)$ is the minimal eigenvalue of $Q_s$). Hence, there exists $M>0$ such that, for $t \ge 1$,
\begin{gather*}
\langle Q_t^{-1}(e^{tB} x - y), e^{tB} x - y \rangle \le \frac{1}{\lambda(t)} |e^{tB} x - y|^2 \le 
\frac{M}{\lambda(1)} (|x|^2  + |y|^2). 
\end{gather*}
 Then $p_t(x,y) \ge \exp(- \frac{M}{2\lambda(1)} (|x|^2  + |y|^2))
 $ $
 \frac{ 1 } { \sqrt{(2 \pi)^N \det(Q_t)} }  $, $t \ge 1$, and 
  \eqref{dd} holds if (ii) is satisfied. 

\hh {${\bf   (i)  \Rightarrow (iv)}$  
   The proof of this  assertion is inspired by  \cite{getoor} and uses  also the Liouville-type theorem for bounded harmonic function proved in \cite{priola_zabczyk}.
   
 \vv   Let us fix a  non-empty open set $O \subset \R^N$ and consider the function
$ \phi_O = \phi : \R^N \to [0,1]$ (cf. \eqref{sst}),    
    $\phi(x) = \P (\tau^x_O < \infty )$, $x\in \R^N$. 
       {\sl We have to prove  that $\phi$ is identically 1.}

  Using the OU semigroup $(P_t)$ we first  check that 
\begin{equation}\label{exc}
P_r \phi (x) \le \phi (x),\;\; r \ge 0,\;\; x \in \R^N.   
\end{equation}
This is a known fact.    We briefly recall the proof for the sake of completeness.   Let us fix $x \in \R^N$ and $r>0 $ and note that $\phi$ is a Borel and bounded function. 
  Since
$\P  (X_{t+r}^x \in O,\;\; \text{for some} \; t  \ge 0 ) $ $\le  
  \P( X_{t}^x \in O,\;\; \text{for some} \; t  \ge 0 ) = \phi(x)$, we get \eqref{exc} 
 by the Markov property:
\begin{gather*} 
  \P (  X_{t+r}^x \in O,\; \text{for some} \; t  \ge 0 )= 
 \E \big [ \E [1_{\{ X_{t+r}^x \in O,\; \text{for some} \; t  \ge 0 \} }    \setminus {\mathcal F}_r] \big ] \\ = \E  \big  [   \phi ( X_{r}^x) \big] 
  = P_r \phi (x).  
\end{gather*}  
 Now take any decreasing sequence $(r_n)$ of positive numbers converging to 0, i.e., $r_n \downarrow 0$. We have
   $\{ X_{t}^x \in O,\;\; \text{for some} \; t  \ge 0 \} = $ $\cup_{n \ge 1}
    \{ X_{t+r_n}^x \in O,\;\; \text{for some} \; t  \ge 0 \}$ (increasing union) and so 
    $\P ( X_{t+r_n}^x \in O,\;\; \text{for some} \; t  \ge 0)$  $ = P_{r_n} \phi(x)     \uparrow \phi(x)$.   Hence
\begin{equation}\label{wee}
P_s \phi(x)  \uparrow \phi(x),\;\; \text{as $s  \to 0^+$,\ $x \in \R^N.$ }
\end{equation}    
 Since $\phi \ge 0$,  properties  \eqref{exc} and \eqref{wee} say that $\phi$ is an excessive function. 
 
 Let us fix $s>0$ and introduce
  the non-negative function $f_s =  \frac{(f - P_{s} \phi )}{s}$. We have
 \begin{gather} \label{2w}
0 \le U f_s (x)  = \frac{1}{s}\int_{0}^{s}
 P_t {\phi} (x) dt < \infty,\;\;\; x \in \R^N.
\end{gather}
Indeed, for any $T>s$,  
   \begin{gather*}
0 \le \frac{1}{s} \int_{0}^{T}P_t ({\phi} - P_{s}{\phi}) (x) dt = \frac{1}{s}\int_{0}^{T}P_t {\phi}(x) dt 
- \frac{1}{s} \int_{0}^{T} P_{t+ \, s}{\phi} (x) dt
\\
= \frac{1}{s} \int_{0}^{T}P_t {\phi}(x) dt 
- \frac{1}{s} \int_{s}^{T + s} P_{t}{\phi} (x) dt 
= 
\frac{1}{s} \int_{0}^{s}P_t {\phi}(x) dt 
- \frac{1}{s} \int_{T}^{T + s} P_{t}{\phi} (x) dt 
\\
\le \frac{1}{s} \int_{0}^{s}P_t {\phi}(x) dt 
\end{gather*} 
(in the last passage we have used that ${\phi} \ge 0$). Passing to the limit as $T \to \infty$ we get \eqref{2w}.   Now  by the Fubini theorem, for any 
$ s > 0$, 
\begin{gather*}
\infty > U f_s(x) = \int_{0}^{\infty} dt \int_{\R^N} f_s(y) p_t(x,y) dy \ge 
 \int_{\R^N} f_s (y)  \big (\int_{1}^{\infty}  p_t(x,y) dt)dy.
\end{gather*}
Since we know   \eqref{dd}   we deduce that  
  $f_s =0$, a.e. on $\R^N$. This  means that, for any $s \ge 0$, 
 \begin{gather}
 {\phi}(x) = P_{s} {\phi}(x),\;\; \text{for any $x \in \R^N$ \, a.e.}.
\end{gather} 
It follows that, for any $t > 0$, 
\begin{equation}\label{ee}
P_t {\phi}(x) = P_t(P_s {\phi})(x) =  P_s (P_t {\phi})(x),  \;\; s\ge 0,
\end{equation}
holds, for any $x \in \R^N$ (not only $a.e.$).
 Thus, for any $t>0$, $P_t {\phi}$ is a bounded harmonic function for $(P_t)$. 
 By hypothesis (HR) and Theorem 3.1 in \cite{priola_zabczyk} we deduce  that $P_t {\phi} \equiv c_t$
 for some constant $c_t $.
 
 Since ${\phi}$ is excessive we know that $P_t {\phi}(x)  \uparrow {\phi} (x)$ as $t \to 0^+$, $x \in \R^N$. It follows that $c_t \uparrow c_0$ and ${\phi} \equiv c_0$.  Take $z \in O$. We have $\phi (z) =1$. Hence $\phi$ is identically 1 and the proof is complete.    
 }
    \end{proof}

\section*{Acknowledgment}
We would like to warmly thank the anonymous referee whose criticism to the first version of the paper led us to
strongly improve our results.

The authors  have been partially supported by the Gruppo Nazionale per l'Analisi Matematica, la Probabilit\`a e le
loro Applicazioni (GNAMPA) of the Istituto Nazionale di Alta Matematica (INdAM).

\bibliographystyle{alpha} 
\bibliography{bibliografia}

\end{document}